\IfFileExists{\jobname.bib}{\immediate\write18{bibtex \jobname}}{}

\documentclass[11pt,a4paper]{article}
\usepackage{geometry}

\usepackage[
    backend=bibtex,
    backref=false,
    url=false,
    doi=false,
    isbn=false,
    eprint=true
    ]{biblatex}

\renewbibmacro{in:}{}  %
\AtEveryBibitem{\clearfield{pages}}  %

\DeclareSourcemap{
  \maps[datatype=bibtex]{
    \map{
      \step[fieldsource=doi,final]
      \step[fieldset=url,null]
    }  
  }
}

\usepackage{xurl}

\usepackage{hyperref}[
    bookmarks=true,
    linktocpage=true,
    bookmarksnumbered=true,
    breaklinks=true,
    pdfstartview=FitH,
    hyperfigures=false,
    plainpages=false,
    naturalnames=true,
    colorlinks=true,
    pagebackref=false,
    pdfpagelabels
]

\hypersetup{
	colorlinks,
	citecolor=blue,
	filecolor=blue,
	linkcolor=blue,
	urlcolor=blue
}

\usepackage{amsmath, amsthm, amssymb}
\usepackage{tikz-cd}
\usepackage[capitalize,noabbrev]{cleveref}
\usepackage{thm-restate}
\usepackage[marginpar]{todo}
\usepackage{mathtools}

\theoremstyle{plain}
\newtheorem{thm}{Theorem}[section]
\Crefname{thm}{Theorem}{Theorems}
\newtheorem{cor}[thm]{Corollary}
\newtheorem{lem}[thm]{Lemma}
\Crefname{lem}{Lemma}{Lemmas}
\newtheorem{prop}[thm]{Proposition}
\Crefname{prop}{Proposition}{Propositions}

\theoremstyle{definition}
\newtheorem{defi}[thm]{Definition}
\newtheorem{ex}[thm]{Example}

\DeclareMathOperator{\im}{im}
\DeclareMathOperator{\piv}{pivot}
\DeclareMathOperator{\pivsimp}{pivot}
\DeclareMathOperator{\pivs}{pivots}
\DeclareMathOperator{\cols}{cols}
\DeclareMathOperator{\id}{id}
\DeclareMathOperator{\Rips}{Rips}
\DeclareMathOperator{\supp}{supp}

\title{Efficient Computation of Image Persistence}

\author{Ulrich Bauer \and Maximilian Schmahl}

\begin{document}

\maketitle

\begin{abstract}
We present an algorithm for computing the barcode of the image of a morphism in persistent homology induced by an inclusion of filtered finite-dimensional chain complexes. 
The algorithm makes use of the clearing optimization and can be applied to inclusion-induced maps in persistent absolute homology and persistent relative cohomology for filtrations of pairs of simplicial complexes. 
The clearing optimization works particularly well in the context of relative cohomology, and using previous duality results we can translate the barcodes of images in relative cohomology to those in absolute homology.
This forms the basis for our implementation of image persistence computations for inclusions of filtrations of Vietoris--Rips complexes in the framework of the software Ripser.
\end{abstract}

\section{Introduction}
\emph{Persistent homology}, the homology of a filtration of simplicial complexes, is a cornerstone in the foundations of topological data analysis. The most common setting studied in the literature is as follows: Given a filtration of simplicial complexes 
\[
K_0\subseteq K_1\subseteq\dots\subseteq K_{N}=
K,
\]
by applying homology with coefficients in a field to each space and to each inclusion map
one obtains a diagram of vector spaces 
\[
H_{*}(K_{\bullet})\colon %
H_{*}(K_0)\to\dots\to H_{*}(K). %
\]
Such a diagram is called a \emph{persistence module}, and it decomposes into a direct sum of indecomposable diagrams, each supported on an interval.
The collection of these intervals, called the \emph{persistence barcode}, has proven to be a powerful invariant of the filtration. 

A natural question to ask is how the barcode changes when the filtration changes.
This leads to the seminal \emph{stability theorem} of Cohen-Steiner et al.\@ \cite{MR2279866}, which asserts that passing from filtrations to barcodes is a $1$-Lipschitz map. One way to approach the stability theorem is via \emph{induced matchings}, which were introduced in \cite{MR3333456}. Given a morphism of filtrations $f_{\bullet}\colon L_{\bullet}\to K_{\bullet}$, the homology functor induces a morphism $H_{*}(f_{\bullet}) \colon H_{*}(L_{\bullet}) \to H_{*}(K_{\bullet})$ of persistence modules.
From this morphism, the induced matching construction yields a partial bijection between the barcodes of $H_{*}(L_{\bullet})$ and $H_{*}(K_{\bullet})$, which can be used to bound the distance between these two barcodes from above.
The induced matching 
is defined in terms of the \emph{image persistence} barcode, i.e., the barcode of $\im H_{*}(f_{\bullet})$,
motivating the problem of computing this barcode.
A first algorithm for this problem has been given by Cohen-Steiner et al.\@ \cite{MR2807543} for the special case where $f_{\bullet}$ is of the form $L_{\bullet} = K_{\bullet} \cap L \hookrightarrow K_{\bullet}$ for some fixed subcomplex $L \subseteq K$.
The authors also propose a method for getting rid of the intersection assumption using a mapping cylinder construction.
This construction, however, might not be computationally feasible and we will show that this additional step is in fact not necessary.

Cohen-Steiner et al.~\cite{MR2807543} propose applications of image persistence for recovering the persistent homology of a noisy function on a noisy domain, see also the related work by Chazal et al.~\cite{MR2846177}.
Recently, Reani and Bobrowski~\cite{reani2021cycle} proposed a method that includes the computation of induced matchings in order to pair up common topological features in different data sets, with applications to statistical bootstrapping. 
Furthermore, the computation of image barcodes is used in a distributed algorithm for persistent homology based on the Mayer--Vietoris spectral sequence by \'Alvaro Torras Casas~\cite{casas2020distributing}.
Image persistence of endomorphisms such as Steenrod squares on the persistent (co)homology of a single filtration has also been proposed by Lupo et al.~\cite{lupo2021persistence} as a tool to get more comprehensive invariants than the standard persistent (co)homology barcodes.

Despite the usefulness of image persistence, there are a few aspects that have prevented these techniques from being widely used in applications so far.
Specifically, to the best of our knowledge, there is no publicly available implementation at this moment.
Furthermore, computation using the known algorithms is slow in comparison to modern algorithms for a single filtration.
Indeed, computing usual persistent homology for larger data sets arising in real-world applications only became feasible in recent years due to optimizations that exploit various structural properties and algebraic identities of the problem \cite{Chen.2011,MR2854319,MR4298669}.
Our goal is to adapt these speed-ups to the computation of images and induced matchings.

The basic algorithm for computing persistent homology is based on performing \emph{matrix reduction}, a variant of column-wise Gaussian elimination, on a \emph{boundary matrix} associated to the given filtration of simplicial complexes. 
This algorithm can be made faster by using the \emph{clearing} optimization, introduced %
by Chen and Kerber in \cite{Chen.2011},
and also used implicitly in the cohomology algorithm by de Silva et al.\@ \cite{MR2854319}. 
In short, this optimization makes use of the homological grading of the boundary matrix and works by deleting certain unnecessary columns during the reduction process.
The basic algorithm for image persistence also requires the reduction of a permuted boundary matrix, to which clearing cannot be straightforwardly applied. 
We will remedy this by showing that one can delete the columns in the permuted boundary matrix that were already reduced to 0 in the boundary matrix corresponding to the codomain filtration.

The clearing optimization works particularly well in conjunction with cohomology based algorithms. These were first studied by de Silva et al.\@ in \cite{MR2854319} for the single filtration case and justified by certain duality results that provide a translation between barcodes for persistent homology and for \emph{persistent cohomology}, 
as well as the barcodes for \emph{persistent relative homology} 
$
H_{*}(K,K_{\bullet})\colon H_{*}(K,K_0)\to\dots\to H_{*}(K,K)
$
and similarly for \emph{persistent relative cohomology}.
We recently extended these duality results \cite{part1} in order to also provide translations for images %
of $H_{*}(f_{\bullet})$ and $H^{*}(f_{\bullet})$, as well as their relative counterparts $H_{*}(f, f_{\bullet})$ and $H^{*}(f, f_{\bullet})$. This allows us to perform cohomology based computations and still obtain the desired barcodes in homology. 

To apply clearing in the relative cohomology setting for image persistence, we will reformulate the algorithm for image persistence by Cohen-Steiner et al.~\cite{MR2807543} in the purely algebraic setting of filtered chain complexes of vector spaces. 
More precisely, we will consider two filtrations of (co)chain complexes $C_{\bullet}$ and $C'_{\bullet}$ and a monomorphism $\varphi_{\bullet}\colon C_{\bullet}\to C'_{\bullet}$.
This setup includes both the absolute homology case $C_{*}(L_{\bullet})\hookrightarrow C_{*}(K_{\bullet})$ and the relative cohomology case $C^{*}(K,K_{\bullet})\hookrightarrow C^{*}(L,L_{\bullet})$ from before. 
The general idea for computing the image %
of $H_{*}(\varphi_{\bullet})$ is to first write it as a subquotient of $C'_{\bullet}$:
\[
\im H_{*}(\varphi_{\bullet})\cong \frac{\varphi_{\bullet}(Z_{*}(C_{\bullet}))}{\varphi_{\bullet}(Z_{*}(C_{\bullet}))\cap B_{*}(C'_{\bullet})},
\]
where the intersection of persistence modules is to be interpreted indexwise, meaning that $(\varphi_{\bullet}(Z_{*}(C_{\bullet}))\cap B_{*}(C'_{\bullet}))_{t} = \varphi_{t}(Z_{*}(C_{t}))\cap B_{*}(C'_{t})$.

Performing matrix reductions that make use of the clearing optimization, we will find \emph{filtration compatible bases} for the filtrations appearing in the equation above such that the inclusion of the filtrations into each other are induced by inclusions of their filtration compatible bases. Filtration compatible bases provide a formal framework for many standard arguments for barcode computations via matrix reduction, and they can be interpreted as special cases of \emph{matching diagrams}, which are equivalent to barcodes \cite{Bauer2020Persistence}. Using the general theory of matching diagrams, the data we compute can easily be shown to determine the barcode of $\im H_{*}(\varphi_{\bullet})$. 

Applying these general considerations in the relative cohomology setting and combining this with the translation between relative cohomology and absolute homology from \cite{part1} yields an algorithm for computing the absolute homology image of $f_{\bullet} \colon L_{\bullet} \to K_{\bullet}$ by reducing two coboundary matrices that can be reduced with clearing as summarized in our main result \cref{thm:main}.
An implementation of this method based on Ripser \cite{MR4298669} is publicly available \cite{ripser_image} and we provide some computational benchmarks.
Our software works under the assumption that $L_{\bullet} = \Rips_{\bullet}(X,d)$ and $K_{\bullet} = \Rips_{\bullet}(X,d')$ are filtrations of Vietoris--Rips complexes corresponding to two metrics $d$ and $d'$ on a finite set $X$ that satisfy $d(x,y) \geq d'(x,y)$ for all $x,y \in X$, which ensures that $L_{t} = \Rips_{t}(X,d)$ is a subcomplex of $K_{t} = \Rips_{t}(X,d')$ for all $t$, with the maps $f_{t} \colon L_{t} \to K_{t}$ being given by inclusion. The implementation also makes uses of a version of the emergent and apparent pairs optimizations.

\subsubsection*{Contributions}
\begin{itemize}
	\item We propose the first direct algorithm for computing the image of a map in persistent homology induced by an inclusion of filtrations of simplicial complexes without imposing any restrictions on the subfiltration (\cref{thm:main}).
	\item Our method makes use of variants of the most important optimizations known from usual persistence computations, including 
	\begin{itemize}
		\item clearing (\cref{cor:im_clearing}), 
		\item cohomology based computations (\cref{prop:abs_rel}), and 
		\item apparent pairs (\cref{sec:apparent}).
	\end{itemize}
	\item We provide an implementation in the framework of Ripser (\cite{ripser_image}), with running time and memory usage that are comparable to those of the simple version of Ripser (\cref{table:experiments}).
	\item This enables the use of image persistence and consequently induced matchings in computational settings (such as supervised learning).
\end{itemize}

\subsubsection*{Outline}
First, we review some basics of the theory of persistence modules and their barcodes in \cref{sec:filtration_compatible_bases}, where we also introduce the language of filtration compatible bases. 
We provide some elementary results for these objects, generalizing well-known results from linear algebra to the filtered setting. 
This theory is then used to formulate basic algorithms for computing image persistence of monomorphisms of (co)chain complexes in \cref{sec:main_proof}.
Subsequently, we explain how to apply clearing for image persistence in \cref{subsec:clearing}.
These results are specialized to the setting of simplicial complexes in \cref{sec:summary}, including a discussion of the connection between their absolute homology and relative cohomology, leading to our main result.
Adapting the emergent and apparent pairs optimizations is discussed in \cref{sec:apparent}.
We finish with some computational benchmarks in \cref{sec:benchmarks}.

\subsubsection*{Notation and Conventions}
Throughout the paper, we fix a totally ordered set $(T,\leq)$ to be $\{0,\dots,N\}$ with the obvious order and a field $\mathbb{F}$ over which all vector spaces are considered.

\section{Linear Algebra for Filtrations}\label{sec:filtration_compatible_bases}
In this section, we develop some machinery based on filtration compatible bases, which forms the foundation for our constructions of image persistence barcodes.
First, we need to recall some basic theory for persistence modules and barcodes. 
We write $\mathbf{Vec}$ for the category of vector spaces over our fixed field $\mathbb{F}$. 
We fixed $T = \{0,\dots,N\}$ as a finite totally ordered index set, and we write $\mathbf{T}$ for $T$ considered as a category.

\begin{defi}\label{defi:persistence_module}
The category of \emph{persistence modules indexed by $\mathbf{T}$} is defined as the category $\mathbf{Vec^{T}}$ consisting of functors $\mathbf{T} \to \mathbf{Vec}$. 
\end{defi}

Since $\mathbf{T}$ is a small category and $\mathbf{Vec}$ is an abelian category, the functor category $\mathbf{Vec^T}$ is again abelian, with kernels, cokernels, images, direct sums, and more generally, all limits and colimits given pointwise. 

The prime example for a persistence module is the persistent homology of a filtration of spaces. Other examples are given by \emph{interval modules}. If $I\subseteq T$ is an interval, the corresponding interval module $C(I)_\bullet$ is
defined by
 \[
C(I)_t=
\begin{cases}
    \mathbb{F} & \text{if } t\in I,\\
    0              & \text{otherwise,}
\end{cases}
\qquad
\text{with structure maps}
\qquad
C(I)_{t,u}=
\begin{cases}
    \id_{\mathbb{F}} & \text{if } t,u\in I,\\
    0              & \text{otherwise.}
\end{cases}
\]

These interval modules are of particular interest because they lead to a structure theory for persistence modules.

\begin{defi}
If there is a family of intervals $(I_{\alpha})_{\alpha \in A}$ such that for a persistence module $M_\bullet$ we have $M_\bullet\cong\bigoplus_{\alpha\in A}C(I_{\alpha})_\bullet$, then $M_\bullet$ is said to have a \emph{barcode} given by $(I_{\alpha})_{\alpha \in A}$. 
\end{defi}

Up to a choice of the index set $A$, barcodes are unique if they exist by a version of the Krull--Remak--Schmidt--Azumaya Theorem \cite[Theorem 2.7]{MR3524869}. By Crawley-Boevey's Theorem \cite{MR3323327}, every persistence module consisting only of finite dimensional vector spaces has a barcode. 

We will be particularly interested in specific kinds of persistence modules where all maps are inclusions.

\begin{defi}
We say that a persistence module $M_\bullet$ is a \emph{filtration} of the vector space $M = M_N$ if $M_{t} \subseteq M_{u}$ for all $t\leq u$ and the structure maps $M_{t,u}$ are given by the subspace inclusions.
For any $m \in M$, we define its \emph{support} in $M_\bullet$ as 
$
\supp_{M_\bullet}(m)=\{t\in T\mid m\in M_{t})\}.
$
A basis $\mathfrak{M}$ of $M$ will be called \emph{filtration compatible} if $\mathfrak{M}_{t} = \mathfrak{M} \cap M_{t}$ is a basis for $M_{t}$ for all $t \in T$.
If $(\mathfrak{M},\leq)$ is an ordered basis for $M$, we say that it is a \emph{filtration compatible ordered basis} if it is filtration compatible and $m\leq m'\in\mathfrak{M}$ implies $\supp m'\subseteq\supp m$.
\end{defi}

If $M_{\bullet}$ and $M'_{\bullet}$ are filtrations of vector spaces, we write $M_{\bullet} \subseteq M'_{\bullet}$ if $M_{t} \subseteq M'_{t}$. We write $M'_{\bullet}/M_{\bullet}$ for the persistence module given by $(M'_{\bullet}/M_{\bullet})_{t} = M'_{t}/M_{t}$. Similarly, if $M''_{\bullet}$ is another filtration with $M''_{\bullet} \subseteq M'_{\bullet}$, we write $M_{\bullet}\cap M''_{\bullet}$ for the persistence module given by $(M_{\bullet} \cap M''_{\bullet})_{t} = M_{t} \cap M''_{t}$.

Observe that if $M_{\bullet}$ is a filtration of vector spaces and $\mathfrak{M}$ is a filtration compatible basis, then $(\supp(m))_{m \in \mathfrak{M}}$ is a barcode of $M_{\bullet}$. 
By interpreting $\mathfrak{M}$ as a so-called matching diagram, this may be seen as a special case of the general equivalence of matching diagrams and barcodes \cite{Bauer2020Persistence}.
This theory also yields the following result that forms the basis for our computational results.

\begin{prop}\label{prop:quotient_barcode}
Let $M_{\bullet} \subseteq M'_{\bullet}$ be filtrations of vector spaces with respective filtration compatible bases $\mathfrak{M}$ and $\mathfrak{M}'$ related by an inclusion $\mathfrak{M} \subseteq \mathfrak{M}'$. Then $M'_{\bullet}/M_{\bullet}$ has the barcode
\[
(\supp_{M'_{\bullet}}(m)\setminus \supp_{M_{\bullet}}(m))_{m \in \mathfrak{M}}\cup(\supp_{M'_{\bullet}}(m))_{m \in \mathfrak{M}'\setminus\mathfrak{M}}.
\]
\end{prop}

We now develop some helpful facts about filtration compatible bases.
We start with a lemma relating supports of basis elements with filtration compatibility.

\begin{lem}\label{lem:filtration_basis_exchange}
Let $M_{\bullet}$ be a filtration of the vector space $M$ with filtration compatible basis~$\mathfrak{M}$. Let $\mathfrak{M}'$ be another basis for $M$ such that there exists a bijection $g\colon\mathfrak{M}\to\mathfrak{M}'$ with $\supp_{M_{\bullet}}(m)=\supp_{M_{\bullet}}(g(m))$ for all $m\in\mathfrak{M}$. Then $\mathfrak{M}'$ is a filtration compatible basis for $M_{\bullet}$.
\end{lem}
\begin{proof}
We need to show that $\mathfrak{M}'_{t}$ is a basis for $M_{t}$ for all $t$. Since $\mathfrak{M}'$ is a basis, $\mathfrak{M}'_{t} = \mathfrak{M} \cap M_{t}$ is linearly independent. We assume $\mathfrak{M}_{t}$ to be a basis for $M_{t}$, so it suffices to show that $\mathfrak{M}_{t}$ and $\mathfrak{M}'_{t}$ have the same cardinality.

To see this, note that if $m \in \mathfrak{M}_{t} = \mathfrak{M} \cap M_{t}$, we must have $t \in \supp_{M_{\bullet}}(m) = \supp_{M_{\bullet}}(g(m))$,
so that $g(m) \in \mathfrak{M}'_{t} = \mathfrak{M}' \cap M_{t}$ holds. Thus, $g$ restricts to a map $\mathfrak{M}_{t} \to \mathfrak{M}'_{t}$. 
Similarly, the restriction of $g^{-1}$ to $\mathfrak{M}'_{t}$ yields a map $\mathfrak{M}'_{t} \to \mathfrak{M}_{t}$.
As the restrictions of $g$ and $g^{-1}$ are inverse to each other, $\mathfrak{M}_{t}$ and $\mathfrak{M}'_{t}$ indeed have the same cardinality.
\end{proof}

Next, we prove a filtered version of a standard fact about intersections of vector spaces.

\begin{lem}\label{lem:intersection}
Let $M'_{\bullet},M''_{\bullet} \subseteq M_{\bullet}$ be filtrations of vector spaces and let $\mathfrak{M}'$ and $\mathfrak{M}''$ be filtration compatible bases for $M'_{\bullet}$ and $M''_{\bullet}$, respectively, such that $\mathfrak{M}' \cup \mathfrak{M}''$ is linearly independent.
Then
$\mathfrak{M}' \cap \mathfrak{M}''$ is a filtration compatible basis for $M'_{\bullet} \cap M''_{\bullet}$.
Moreover, for all $m \in \mathfrak{M}' \cap \mathfrak{M}''$
\[
\supp_{M'_{\bullet} \cap M''_{\bullet}}(m)=\supp_{M'_{\bullet}}(m)\cap\supp_{M''_{\bullet}}(m) .
\]
\end{lem}
\begin{proof}
We want to show that $\mathfrak{M}' \cap \mathfrak{M}''$ is a filtration compatible basis for $M'_{\bullet} \cap M''_{\bullet}$, i.e., 
\[
(\mathfrak{M}' \cap \mathfrak{M}'')_{t} = \mathfrak{M}' \cap \mathfrak{M}'' \cap M'_{t} \cap M''_{t} = \mathfrak{M}'_{t} \cap \mathfrak{M}''_{t}
\] 
is a basis for $(M'_{\bullet} \cap M''_{\bullet})_{t} = M'_{t} \cap M''_{t}$ for all $t \in T$. 
By standard linear algebra, 
for subspaces $V',V''\subseteq V$ of some vector space $V$ with respective bases $\mathfrak V'$ and $\mathfrak V''$, the intersection $\mathfrak V' \cap \mathfrak V''$ of the bases is a basis for the intersection $V' \cap V''$ of spaces if the union $\mathfrak V' \cup \mathfrak V''$ of the bases is linearly independent.
In particular, 
$\mathfrak{M}'_{t} \cap \mathfrak{M}''_{t}$ is a basis for $M'_{t} \cap M''_{t}$ since the union $\mathfrak{M}'_{t} \cup \mathfrak{M}''_{t} \subseteq \mathfrak{M}' \cup \mathfrak{M}''$ is linearly independent by assumption. 
Thus, $\mathfrak{M}' \cap \mathfrak{M}''$ is a filtration compatible basis for $M' \cap M''$.

For the supports, we have
\begin{align*}
\supp_{M'_{\bullet} \cap M''_{\bullet}}(m) 
&= \{t \in T \mid m \in M'_{t} \cap M''_{t}\} \\
&= \{t \in T \mid m \in M'_{t}\} \cap \{t \in T \mid m \in M''_{t}\} \\
&= \supp_{M'_{\bullet}}(m) \cap \supp_{M''_{\bullet}}(m)
\end{align*}
for all $m \in \mathfrak{M}' \cap \mathfrak{M}''$, proving the claim.
\end{proof}

Finally, we state a version of the rank-nullity-theorem for filtrations.

\begin{lem}\label{lem:quotient_basis}
Let $\phi_\bullet\colon M_\bullet \to P_\bullet$ be a morphism of filtrations of vector spaces, let $\mathfrak{M}$ be a filtration compatible basis for $M_\bullet$, let $\mathfrak{M}' = \mathfrak{M} \cap \ker\phi$, and assume that $\mathfrak{M}'' = (\phi(m))_{m \in \mathfrak{M}\setminus\mathfrak{M}'}$ is a linearly independent family of vectors.
	Then $\mathfrak{M}'$ is a filtration compatible basis for $\ker\phi_\bullet$, and $\mathfrak{M}'' $ is a filtration compatible basis for $\im\phi_\bullet$.
Moreover, 
\[
\supp_{\ker\phi_\bullet}(m')=\supp_{M_\bullet}(m')
\quad\text{and}\quad
\supp_{\im\phi_\bullet}(\phi(m))=\supp_{M_\bullet}(m)
\]
for all $m'\in\mathfrak{M}'$ and for all $m \in \mathfrak{M}\setminus\mathfrak{M}'$.
\end{lem}
\begin{proof}
We show that $\mathfrak{M}' = \mathfrak{M} \cap \ker\phi$ is a filtration compatible basis for $\ker\phi_{\bullet}$.
Because $M_{\bullet}$ is a filtration, we have $\ker\phi_{t} = M_{t} \cap \ker\phi$ for all $t\in T$. 
Denoting the constant filtration of $\ker\phi$ by $\Delta\ker\phi$, we have $\ker\phi_{\bullet} = M_{\bullet} \cap \Delta\ker\phi$. 
By standard linear algebra, $\mathfrak{M}'' = (\phi(m))_{m \in \mathfrak{M}\setminus\mathfrak{M}'}$ being linearly independent implies that $\mathfrak{M'}$ is a basis for $\ker\phi$. 
Hence, $\mathfrak{M}'$ is a filtration compatible basis for the constant filtration $\Delta\ker\phi$.
By assumption, $\mathfrak{M} \cup \mathfrak{M}' = \mathfrak{M}$ is linearly independent, so we can apply \cref{lem:intersection} to obtain that the intersection $\mathfrak{M} \cap \mathfrak{M}' = \mathfrak{M}'$ is a filtration compatible basis for the intersection $M_{\bullet} \cap \Delta\ker\phi = \ker\phi_{\bullet}$ satisfying
\[
\supp_{\ker\phi_{\bullet}}(m') = \supp_{\Delta\ker\phi}(m') \cap \supp_{M_{\bullet}}(m') = T \cap \supp_{M_{\bullet}}(m') = \supp_{M_{\bullet}}(m')
\]
for all $m'\in\mathfrak{M}'$.

To show that $\mathfrak{M}''$ is a filtration compatible basis for $\im\phi_{\bullet}$, we need to show that $\mathfrak{M}''_{t} = \mathfrak{M}'' \cap \im\phi_{t}$ is a basis for $\im\phi_{t}$ for all $t \in T$. 
Since we assume $\mathfrak{M}''$ to be linearly independent, $\mathfrak{M}''_{t}$ is linearly independent as well.
It remains to be shown that $\mathfrak{M}''_{t}$ is also a generating set $\im\phi_{t}$. 
Because $\mathfrak{M}_{t}$ is a basis for $M_{t}$ by assumption, the family $(\phi_{t}(m))_{m \in \mathfrak{M}_{t}\setminus\ker\phi_{t}}$ generates $\im\phi_{t}$.
Hence, it suffices to show that 
\[
(\phi_{t}(m))_{m \in \mathfrak{M}_{t}\setminus\ker\phi_{t}} \subseteq (\phi(m))_{m \in \mathfrak{M}\setminus\mathfrak{M}'} \cap \im\phi_{t} = \mathfrak{M}''_{t}.
\]
So let $m \in \mathfrak{M}_{t}\setminus\ker\phi_{t}$. 
We have $\phi_{t}(m) \neq 0 \in P_{t}$, so since $P_{\bullet}$ is a filtration we also get $\phi(m) \neq 0 \in P$. 
Thus, $m\in \mathfrak{M}_{t} \setminus \ker\phi \subseteq \mathfrak{M} \setminus \ker\phi = \mathfrak{M}\setminus\mathfrak{M}'$.
This implies $\phi_{t}(m) = \phi(m) \in \mathfrak{M}''_{t}$ as needed, so $\mathfrak{M}''_{t}$ generates $\im\phi_{t}$ and hence forms a basis for $\im\phi_{t}$.
Bases are minimal generating sets, so we obtain an equality $(\phi_{t}(m))_{m \in \mathfrak{M}_{t}\setminus\ker\phi_{t}} = (\phi(m))_{m \in \mathfrak{M}\setminus\mathfrak{M}'} \cap \im\phi_{t}$ from the previously shown inclusion. 
Thus, for $m \in \mathfrak{M}\setminus\mathfrak{M}'$ we have $\phi(m) \in \im\phi_{t}$ if and only if $m \in \mathfrak{M}_{t} \subseteq M_{t}$.
This yields $\supp_{\im\phi_{\bullet}}(\phi(m)) = \supp_{M_{\bullet}}(m)$, which finishes the proof.
\end{proof}

Note that if one drops the assumption that the image $\im\phi_{\bullet}$ is a filtration from the setting of the lemma above, it may happen that $\mathfrak{M}' = \mathfrak{M} \cap \ker\phi$ is a basis for $\ker\phi$ but not a filtration compatible basis for the filtration $\ker\phi_{\bullet}$.

\section{Computing Image Persistence Barcodes}\label{sec:computation}
Recall that we fixed a finite totally ordered index set $T = \{0,\dots,N\}$ and a field $\mathbb{F}$ over which we consider vector spaces. For our purposes, a chain (resp.\ cochain) complex is a graded finite dimensional vector space with a differential of degree $-1$ (resp.\ $1$) that squares to $0$. A filtration of (co)chain complexes is a filtration of finite dimensional vector spaces with differentials that commute with the inclusion maps. 
Recall that a basis for the final vector space in a filtration is called filtration compatible if it yields bases for the constituent vector spaces of the filtration by intersecting. Further, recall that if the basis is ordered, we say that it is a filtration compatible ordered basis if its order refines the order in which the basis elements appear in the filtration.

\begin{defi}
If $C_{\bullet}$ is a filtration of the (co)chain complex $C$ with a filtration compatible ordered basis $\mathfrak{C}$, then the matrix $D$ representing the (co)boundary operator on $C$ with respect to $\mathfrak{C}$ is called \emph{filtration (co)boundary matrix}.
\end{defi}

\begin{ex}\label{ex:simp_filt}
If 
\(
K_{\bullet}:\emptyset=K_{0}\subseteq K_{1}\subseteq\dots\subseteq K_{N}=K
\)
is a filtration of finite simplicial complexes, we get a filtration of chain complexes $C_{*}(K_{\bullet})$. A filtration compatible ordered basis is given by the simplices of $K$, ordered by a linear refinement of the order in which they appear in the filtration. If $D^{K}$ is a corresponding filtration boundary matrix, then one can check (see \cite{MR2854319}) that $(D^{K})^{\perp}$ is a filtration coboundary matrix for the filtration of relative cochains $C^{*}(K, K_{\bullet})$, representing the coboundary operator on $C^{*}(K)$ with respect to the dual basis corresponding to the simplices of $K$, ordered by the opposite of the filtration order. Here, $(-)^{\perp}$ denotes taking the transpose of a matrix along its anti-diagonal.
\end{ex}

To avoid notational clutter, we will from now on only talk about chain complexes in the general setting, but everything also straightforwardly applies to cochain complexes.

\begin{defi}
If $X$ is a matrix, we write $x_i$ for the $i$-th column of the matrix $X$. For a non-zero column vector $x_{i}$, we define $\piv x_{i}$ as the largest index where the column has a non-zero entry. We write $\pivs X$ for the set of all indices which occur as pivots of non-zero columns of $X$. A matrix is called \emph{reduced} if no two non-zero columns have the same pivot. 
\end{defi}

Note that any set of non-zero vectors with unique pivots is linearly independent. In particular, the non-zero columns of a reduced matrix are linearly independent.

Computing the barcode for the homology of a filtered persistent chain complex is done by \emph{reducing} a filtration boundary matrix $D$, i.e., performing a variant of Gaussian elimination on the columns of this matrix until one obtains a reduced matrix. 
This can be expressed as finding a reduced matrix $R$ and a full-rank upper-triangular matrix $V$ such that $R = DV$.
The barcode for persistent homology may then be obtained from this data as follows.

\begin{thm}[Cohen-Steiner et al. \cite{MR2389318}]
\label{thm:single_filt}
Let $D$ be a filtration boundary matrix of a filtration of chain complexes $C_{\bullet}$ and assume we have a full-rank and upper-triangular matrix $V$ such that $R=DV$ is reduced. Then $H_*(C_{\bullet})$ has a barcode given by the multiset
\[
\left\{\supp_{C_{\bullet}} (r_j) \setminus \supp_{C_{\bullet}} (v_j)\mid r_j\neq 0\right\}
\cup
\left\{\supp_{C_{\bullet}} (v_i) \mid r_i=0~\text{and} ~  i\notin\pivs R\right\}.
\]
\end{thm}

The theorem is formulated in a different language by Cohen-Steiner et al., but the version above will also follow as a special case from \cref{cor:image_barcode}. Note that the theorem is also compatible with the homological grading, i.e., one gets a barcode for $H_{d}(C_{\bullet})$ by choosing a filtration compatible ordered basis that is also compatible with the grading and then only taking those intervals which come from columns of $D$ corresponding to $d$-dimensional basis elements.

\subsection{Image Barcodes via Matrix Reduction}\label{sec:main_proof}
We now turn to the setting of image persistence.
Let $C_{\bullet}$ and $C'_{\bullet}$ be filtrations of the chain complexes $C$ and $C'$ with corresponding filtration compatible ordered bases $\mathfrak{C}$ and $\mathfrak{C}'$. Let $D$ and $D'$ be the corresponding filtration boundary matrices. Assume that we are given an injection of filtrations $\varphi_{\bullet}\colon C_{\bullet}\to C'_{\bullet}$
 such that the map $\varphi \colon C \to C'$ on the final filtration step is an isomorphism. Let $F$ be the matrix representing $\varphi$ with respect to $\mathfrak{C}$ and $\mathfrak{C}'$ and define $D^{\varphi}=DF^{-1}=F^{-1}D'$. 

\begin{ex}\label{ex:im_simp_filt}
If 
\[
K_{\bullet}:\emptyset=K_{0}\subseteq K_{1}\subseteq\dots\subseteq K_{N}=K
\qquad
\text{and} 
\qquad
L_{\bullet}:\emptyset=L_{0}\subseteq L_{1}\subseteq\dots\subseteq L_{N}=L
\]
are filtrations of finite simplicial complexes, we get filtrations of chain complexes $C_{*}(K_{\bullet})$ and $C_{*}(L_{\bullet})$. If we are given a monomorphism $L_{\bullet}\to K_{\bullet}$ that induces an isomorphism $L \to K$ (e.g. by assuming $L_{i}\subseteq K_{i}$ for all $i$ and $L=K$), then we are in the setting above. Filtration compatible ordered bases are given by the simplices of $K$ and $L$, ordered by a linear refinement of the order in which they appear in the respective filtrations. If $D^{L}$, $D^{K}$ and $D^{f}$ are the corresponding boundary matrices for $C_{*}(L_{\bullet})\to C_{*}(K_{\bullet})$, then analogously to \cref{ex:simp_filt}, we obtain that $(D^{L})^{\perp}$, $(D^{K})^{\perp}$ and $(D^{f})^{\perp}$ are the coboundary matrices for the relative cohomology counterpart $C^{*}(K, K_{\bullet})\to C^{*}(L, L_{\bullet})$.
\end{ex}

Our goal is to determine a barcode for $\im H_{*}(\varphi_{\bullet})$ via matrix reduction, so assume we have $R=DV$ and $R^{\varphi}=D^{\varphi}V^{\varphi}$ reduced with $V$ and $V^{\varphi}$ full-rank and upper-triangular. The columns of $R$, $D$, $V$, $R^{\varphi}$ and $D^{\varphi}$ should be interpreted as coordinate vectors with respect to $\mathfrak{C}$ and the columns of $V^{\varphi}$ as coordinate vectors with respect to $\mathfrak{C}'$. Recall that if $X$ is a matrix, we denote its $j$th column by $x_j$. 
The main result can then be stated as follows.

\begin{thm}\label{cor:image_barcode}
The image of $H_{*}(\varphi_{\bullet})$ has a barcode given by the multisets
\[
\left\{\supp_{C_{\bullet}} (r^{\varphi}_j) \setminus \supp_{C'_{\bullet}} (v^{\varphi}_j)\neq \emptyset\mid r^{\varphi}_j\neq 0\right\}
\cup
\left\{\supp_{C_{\bullet}} (v_i) \mid r^{\varphi}_{i}=0~\text{and} ~ i\notin\pivs R^{\varphi}\right\}.
\]
\end{thm}

The supports of column vectors that appear in the theorem can easily be determined from the initial data via pivots: If $\mathfrak{M}$ is a filtration compatible ordered basis for a filtration $M_{\bullet}$, then we can consider elements of the filtered vector space $M$ via their coordinate vectors with respect to $\mathfrak{M}$. Because $\mathfrak{M}$ is a filtration compatible ordered basis, we then have $\supp_{M_{\bullet}}(v)=\supp_{M_{\bullet}}(v')$ if and only if $\piv v=\piv v'$ for any two such coordinate vectors $v$ and $v'$. In particular, this means that in the setting of simplicial complexes we can determine the support of a column vector from the support of its pivot simplex.

The proof of \cref{cor:image_barcode} will be based on a sequence of intermediate results. 
As alluded to in the introduction, the general idea is to write 
\[
\im H_{*}(\varphi_{\bullet})\cong \frac{\varphi(Z_{*}(C_{\bullet}))}{\varphi(Z_{*}(C_{\bullet}))\cap B_{*}(C'_{\bullet})},
\]
find filtration compatible bases for the filtrations appearing on the right that include into each other and then apply \cref{prop:quotient_barcode}.

If $X$ is a matrix, we will write $\cols X$ for the family of all its non-zero column vectors.

\begin{lem}\label{lem:boundary_basis}
The family $\cols FR^{\varphi}$ is a filtration compatible basis for $B_{*}(C'_{\bullet})$, and for all $j$ with $r^{\varphi}_{j}\neq 0$ we have 
\[
\supp_{B_{*}(C'_{\bullet})}(Fr^{\varphi}_{j})=\supp_{C'_{\bullet}}(v^{\varphi}_{j}) .
\]
\end{lem}
\begin{proof}
We start by showing that $\cols V^{\varphi}$ is a filtration compatible basis for $C'_{\bullet}$:
We have $\piv v^{\varphi}_{j}=j$ since $V^{\varphi}$ is full-rank and upper-triangular. 
It follows that $v^{\varphi}_{j}$ has the same support in $C'_{\bullet}$ as the $j$th element of $\mathfrak{C}'$. 
Thus, $\cols V^{\varphi}$ is a filtration compatible basis for $C'_{\bullet}$ by \cref{lem:filtration_basis_exchange}.

Next, note that $(\partial(v))_{v \in \cols V^{\varphi}\setminus\ker\partial} = \cols FR^{\varphi}$ is linearly independent since $R^{\varphi}$ is reduced and $F$ has full rank. 
Thus, we can apply \cref{lem:quotient_basis} to the map of filtrations $\partial_{\bullet} \colon C'_{\bullet} \to C'_{\bullet}$ and the filtration compatible basis $\cols V^{\varphi}$ to obtain that $\cols FR^{\varphi}$ is a filtration compatible basis for $B_{*}(C'_{\bullet}) = \im\partial_{\bullet}$. 
The assertion on the supports also follows from the support formula in \cref{lem:quotient_basis}.
\end{proof}

Now that we have a filtration compatible basis for $B_{*}(C'_{\bullet})$, we want to extend it to a filtration compatible basis for $\varphi_{\bullet}(Z_{*}(C_{\bullet}))$.

\begin{lem}\label{lem:z_basis}
The family
\[
\mathfrak{Z}=\cols FR^{\varphi}\cup \left\{Fv_{j}\mid r_{j}=0~\text{and} ~ j\notin\pivs R^{\varphi}\right\}.
\]
is a filtration compatible basis for $\varphi_{\bullet}(Z_{*}(C_{\bullet}))$, and for all $z\in\mathfrak{Z}$ we have
\[
\supp_{\varphi_{\bullet}(Z_{*}(C_{\bullet}))}(z)=\supp_{C_{\bullet}}(F^{-1}z) .
\]
\end{lem}
\begin{proof}
We start by showing that 
\[
\mathfrak{X}=\cols R^{\varphi}\cup \left\{v_{j}\mid j\notin\pivs R^{\varphi}\right\}.
\]
is a filtration compatible basis for $C_{\bullet}$.
The same argument as in the beginning of the proof of \cref{lem:boundary_basis} yields that $\cols V$ is a filtration compatible basis for $C_{\bullet}$.
Next, note that $\mathfrak{X}$ is linearly independent since all elements have unique pivots: $R^{\varphi}$ is reduced and we only consider those $v_{j}$ with $\piv v_{j}=j\notin\pivs R^{\varphi}$. Moreover, we have a bijection $\mathfrak{X}\to\cols V$ given by mapping $v_{j}$ to itself and mapping $r^{\varphi}_{j}$ to $v_{\piv r^{\varphi}_{j}}$. Recall that $\piv r^{\varphi}_{j}=\piv v_{\piv r^{\varphi}_{j}}$ implies 
\[
\supp_{C_{\bullet}}\left(r^{\varphi}_{j}\right)=\supp_{C_{\bullet}}\left(v_{\piv r^{\varphi}_{j}}\right).
\]
Since $\cols V$ is a filtration compatible basis for $C_{\bullet}$, \cref{lem:filtration_basis_exchange} now implies that $\mathfrak{X}$ is also a filtration compatible basis for $C_{\bullet}$.

Next, we can apply \cref{lem:quotient_basis} to the boundary operator $\partial_{\bullet} \colon C_{\bullet} \to C_{\bullet}$ and the filtration compatible basis $\mathfrak{X}$ since $(\partial(v))_{v\in\mathfrak{X}\setminus\ker\partial} \subseteq \cols R$ is linearly independent by reducedness of $R$. 
We obtain that $\mathfrak{X} \cap \ker\partial = F^{-1}\mathfrak{Z}$ is a filtration compatible basis for $\ker\partial_{\bullet} = Z_{*}(C_{\bullet})$ with $\supp_{Z_{*}(C_{\bullet})}(x) = \supp_{C_{\bullet}}(x)$ for all $x \in F^{-1}\mathfrak{Z}$. 
The claim now follows from the fact that $\varphi_{\bullet}$ is mono, so that its restriction is an isomorphism $Z_{*}(C_{\bullet})\to \varphi_{\bullet}(Z_{*}(C_{\bullet}))$ represented by $F$.
\end{proof}

Having extended the filtration compatible basis for $B_{*}(C'_{\bullet})$ to one for $\varphi_{\bullet}(Z_{*}(C_{\bullet}))$, we also obtain one for $\varphi_{\bullet}(Z_{*}(C_{\bullet}))\cap B_{*}(C_{\bullet})$.

\begin{lem}\label{lem:cap_basis}
The set $\mathfrak{B}=\cols FR^{\varphi}$ is a filtration compatible basis for $\varphi_{\bullet}(Z_{*}(C_{\bullet}))\cap B_{*}(C'_{\bullet})$,
and for all $j$ with $r^{\varphi}_{j}\neq 0$ we have 
\begin{align*}
\supp_{\varphi_{\bullet}(Z_{*}(C_{\bullet}))\cap B_{*}(C'_{\bullet})}(Fr^{\varphi}_{j})
&=\supp_{C_{\bullet}}(r^{\varphi}_{j})\cap\supp_{C'_{\bullet}}(v^{\varphi}_{j}) .
\end{align*}
\end{lem}
\begin{proof}
Recall that $\cols FR^{\varphi}$ is a filtration compatible basis for $B_{*}(C'_{\bullet})$ and $\mathfrak{Z}$ is one for $\varphi_{\bullet}(Z_{*}(C_{\bullet}))$. We have $\cols FR^{\varphi}\cup\mathfrak{Z}=\mathfrak{Z} \subseteq C'$, which is linearly independent.
Thus, the claim follows from \cref{lem:intersection} together with the support equalities from \cref{lem:boundary_basis,lem:z_basis}.
\end{proof}

We are now prepared to prove the main result of this section.
\begin{proof}[Proof of \cref{cor:image_barcode}]
By definition of the induced map in homology, we have 
\[
\im H_{*}(\varphi_{\bullet})\cong \frac{\varphi_{\bullet}(Z_{*}(C_{\bullet}))}{\varphi_{\bullet}(Z_{*}(C_{\bullet}))\cap B_{*}(C'_{\bullet})}.
\]
Thus, the claim follows by applying \cref{prop:quotient_barcode} to the inclusion $\varphi_{\bullet}(Z_{*}(C_{\bullet}))\cap B_{*}(C'_{\bullet}) \subseteq B_{*}(C'_{\bullet})$ with the filtration compatible bases $\mathfrak{B}\subseteq \mathfrak{Z}$ with supports as previously determined in \cref{lem:cap_basis,lem:z_basis}, combined with the additional observation that the interval difference
$\supp_{C_{\bullet}} (r^{\varphi}_j) \setminus (\supp_{C_{\bullet}}(r^{\varphi}_{j})\cap\supp_{C'_{\bullet}}(v^{\varphi}_{j}))$ is non-empty only if it equals 
$\supp_{C_{\bullet}} (r^{\varphi}_j) \setminus \supp_{C'_{\bullet}} (v^{\varphi}_j)$.
\end{proof}

\subsection{Clearing}\label{subsec:clearing}
Before discussing clearing in the image setting, we first recall the idea taken from \cite{Chen.2011} for computing the persistent homology of a filtration of chain complexes $C_{\bullet}$ by reducing the boundary matrix $D$ to $R = D \cdot V$. We keep the notation from the previous section, and we assume that our filtration compatible basis $\mathfrak{C}$ is compatible with the grading in the sense that the restriction of this basis to each grading summand is again a basis of that summand. 
Our discussion focuses on chain complexes, but of course the findings naturally apply to \emph{co}chain complexes with the appropriate adjustments to the grading.

In general, a non-zero column $r_j$ implies that $r_{i}=0$ for $i = \piv r_j$. Moreover, the homological degree of the $i$-th element of $\mathfrak{C}$ is one less than that of the $j$-th element. This leads to the clearing procedure: Instead of simply reducing $D$ by column operations from left to right, we reduce columns in decreasing order of their homological degree (increasing in the case of cohomology). Before using column operations to reduce the columns in dimension $d$, we set $r_j=0$ for all indices $j$ which appear as pivots of the already reduced columns in dimension $d+1$.

Returning to the image setting, we also assume that the basis $\mathfrak{C}'$ and the map $\varphi_{\bullet}\colon C_{\bullet}\to C'_{\bullet}$ are compatible with the grading. Here, there is no direct analogue to the procedure outlined above: The mixed basis boundary matrix $D^{\varphi}$ fails to admit the property described above, i.e., it may happen that $r^{\varphi}_j\neq 0$ but $r^{\varphi}_{i}\neq 0$ for $i=\piv r^{\varphi}_j$. In order to obtain a useful condition for columns of $R^{\varphi}$ to be zero, we need to additionally consider the boundary matrix $D'=FD^{\varphi}$ and assume we have a reduction $R'=D'V'$.

\begin{prop}\label{prop:im_clearing}
Let $R'=D'V'$ and $R^{\varphi}=D^{\varphi}V^{\varphi}$ be reduced. 
For all indices $j$ we have $r^{\varphi}_j=0$ if and only if $r'_j=0$.
\end{prop}
\begin{proof}
First, note that $r^{\varphi}_j=0$ if and only if $Fr^{\varphi}_j=0$. Moreover, $FR^{\varphi}$ and $R'$ have the same column space, since $FR^{\varphi}=R'(V')^{-1}V^{\varphi}$. Thus, the number of zero columns of $R^{\varphi}$ is the same as the number of zero columns of $R'$ since their ranks are equal and their non-zero columns are linearly independent. Now, it suffices to show that $r^{\varphi}_j=0$ implies $r'_j=0$, so assume $r^{\varphi}_j=0$. Then $Fr^{\varphi}_j=0$, but $Fr^{\varphi}_j$ is also the same as the $j$-th column of $R'(V')^{-1}V^{\varphi}$. This is a linear combination of columns of $R'$ with non-zero coefficient for $r'_j$ since $(V')^{-1}V^{\varphi}$ is full-rank and upper-triangular. Non-zero columns of $R'$ are linearly independent, so this linear combination can only be zero if $r'_j=0$.
\end{proof}

In order to apply clearing to the reduction of $D^{\varphi}$, one can now reduce $D'$ with clearing as usual, and clear the columns with the same indices in $D^{\varphi}$.
Even more than that, one can also clear all other columns in $D^{\varphi}$ that correspond to 0 columns in the reduction of $D'$, meaning also those that have been reduced to 0 via column operations on $D'$.
With this optimization, there are no more columns that have to be reduced to 0 in the reduction of $D^{\varphi}$ even before any column operations on $D^{\varphi}$ have been performed. 
The only reduction steps that are left to be performed are those that make pivots unique among the non-zero columns.

\begin{cor}\label{cor:im_clearing}
If $D'$ has already been reduced to $R'$, one can set $r^{\varphi}_{j} = 0$ for all $j$ with $r'_{j} = 0$ before reducing $D^{\varphi}$, and no further columns of $R^{\varphi}$ will be reduced to $0$.
\end{cor}

\subsection{Assembling Barcodes from (Co)homology Computations}\label{sec:summary}
Recalling our concrete setting of simplicial complexes and their persistent homology, assume that we are given filtrations of two isomorphic simplicial complexes
\[
K_{\bullet}:\emptyset=K_{0}\subseteq K_{1}\subseteq\dots\subseteq K_{N}=K
\qquad
\text{and}
\qquad
L_{\bullet}:\emptyset=L_{0}\subseteq L_{1}\subseteq\dots\subseteq L_{N}=L
\]
and a monomorphism $f_{\bullet}\colon L_{\bullet}\to K_{\bullet}$ inducing an isomorphism $f \colon L \to K$.
Following the notation from \cref{ex:im_simp_filt} and applying the previous results with $\varphi_{\bullet} = C_{*}(f_{\bullet})$, we see that the barcode of $\im H_{*}(f_{\bullet})$ can be determined via reductions of $D^{L}$ and $D^{f}$ and that the reduction of $D^{f}$ may be performed with clearing if $D^{K}$ has previously been reduced.

As is also known from the single filtration case, clearing needs to be initialized by performing a full persistence computation in the first homological degree for which persistence is computed.
Because persistence computations are often only feasible in low dimensions and practitioners are often only interested in barcodes in low degrees, it is much more powerful to apply clearing for \emph{co}homological grading, allowing for the initialization to be performed in degree 0.
Thus, our goal is to perform cohomological computations and still recover the image $\im H_{*}(f_{\bullet})$ in homology.

As a first step towards that goal, we recall that $\im H^{*}(f_{\bullet})$ and $\im H_{*}(f_{\bullet})$ have the same barcodes \cite{part1}.
However, the persistent cochain complex giving rise to persistent cohomology is not filtered, so the basic matrix reduction algorithm does not directly apply there.
Thus, we instead perform computations in the relative cohomology setting given by the map $H^{*}(f, f_{\bullet})$.
Its image no longer has the same barcode as $\im H_{*}(f_{\bullet})$, but there are some correspondence results \cite[Section 6.2]{part1}, which we will summarize next.
To state the result, for a barcode $B$ we write $B^{\dagger}$ for the intervals in $B$ that do not extend to the endpoints of our index set $T$ and $B^{\infty}$ for those intervals that do.

\begin{prop}[Bauer, Schmahl~\cite{part1}]
\label{prop:abs_rel}
We have
\[
B(\im H_{d-1}(f_{\bullet}))^{\dagger} = B(\im H^{d}(f, f_{\bullet}))^{\dagger}
\]
for all degrees $d$, and the map $I \mapsto T\setminus I$ defines bijections
\begin{align*}
B(\im H_{*}(f_{\bullet}))^{\infty} &\cong B(H^{*}(L, L_{\bullet}))^{\infty}, \\
B(\im H^{*}(f,f_{\bullet}))^{\infty} &\cong B(H_{*}(K_{\bullet}))^{\infty}.
\end{align*}
\end{prop}

Note that none of the intervals in the barcodes we consider span the whole index set $T$ because we assume that our filtrations start with $L_{0} = K_{0} = \emptyset$.

\cref{prop:abs_rel} implies that in order to determine the barcode of $\im H_{*}(f_{\bullet})$, it suffices to compute $B(H^{*}(L, L_{\bullet}))^{\infty}$ and $B(\im H^{*}(f, f_{\bullet}))^{\dagger}$.
Following \cref{ex:simp_filt,thm:single_filt}, we observe that $B(H^{*}(L, L_{\bullet}))^{\infty}$ may be determined from a reduction of the coboundary matrix $(D^{L})^{\perp}$, and following \cref{ex:im_simp_filt,cor:image_barcode}, we know that $B(\im H^{*}(f, f_{\bullet}))^{\dagger}$ may be determined from a reduction of the coboundary matrix $(D^{f})^{\perp}$.
In the relative cohomology setting, the matrices $(D^{L})^{\perp}$ and $(D^{f})^{\perp}$ play the roles of $D'$ and $D^{\im}$ in the general setting, so by \cref{cor:im_clearing} we can simultaneously reduce these matrices with clearing. 

We summarize the discussion in the following theorem.
To simplify notation, we will assume that we are given funtions $k$ and $l$ on $K \cong L$ that induce the filtrations $K_{\bullet}$ and $L_{\bullet}$, respectively, via their sublevel set filtrations. 
For example, the functions $l$ and $k$ would be given by the diameter functions if $K_{\bullet}$ and $L_{\bullet}$ are Vietoris--Rips filtrations for different metrics on the same set of points.
Moreover, recall that if $A$ is a matrix, we write $a_{j}$ for its $j$th column.
To determine barcodes from reductions of boundary matrices, recall that the column and row indices of the matrices $D^{f}$, $(D^{f})^{\perp}$, etc., correspond to the simplices of $K \cong L$ in different orders. 
In particular, the pivot index of a column vector $c$ will in this context always corresponds to a unique simplex, which we also denote by $\pivsimp c$.
Combining \cref{cor:image_barcode,cor:im_clearing,prop:abs_rel}, we now get the following.

\begin{cor}\label{thm:main}
The matrices $(D^{L})^{\perp}$ and $(D^{f})^{\perp}$ can be reduced with clearing, and given reductions $S = (D^{f})^{\perp} W$ and $R = (D^{L})^{\perp} V$,
the barcode of $\im H_{*}(f_{\bullet})$ can be determined as the multiset
\[
\left\{ \left[ l ( \pivsimp w_{j} ), k ( \pivsimp s_{j} ) \right) \neq \emptyset \mid s_{j} \neq 0 \right\}
\cup
\left\{ \left[ l (\pivsimp v_i ), \infty \right) \mid r_i = 0~\text{and} ~  i\notin\pivs R\right\}.
\]
\end{cor}

Recall that the column and row indices of the coboundary matrices indicated by $(-)^{\perp}$ correspond to the simplices of $K \cong L$ in reverse filtration order.
Hence, the pivot simplex of a column vector appearing in the theorem will be the \emph{first} simplex appearing in the filtration among those that correspond to a non-zero entry of the column, while for the usual boundary matrices $D^{L}$, $D^{K}$, etc., the pivot simplex of a column would be the one that appears \emph{last} in the filtration among those simplices that correspond to a non-zero entry.

\subsection{Apparent and emergent pairs in image matrix reduction}\label{sec:apparent}
An important optimization in persistence computation leading to significant computational improvements is given by utilizing the \emph{apparent pairs} in the filtration, which are pairs $(\sigma,\tau)$ in the filtration such that $\sigma$ is the latest facet of $\tau$ in the filtration and $\tau$ is the earliest cofacet of $\sigma$.
Apparent pairs always form persistence pairs, since the corresponding columns are reduced already in the (co)boundary matrix.
More generally, if $(\sigma,\tau)$ is a persistence pair and $\tau$ is the earliest cofacet of $\sigma$, we say that $(\sigma,\tau)$ is an \emph{emergent cofacet pair}.
In Ripser \cite{MR4298669}, the zero persistence emergent pairs are identified during the construction of the columns of the coboundary matrix, terminating this construction early once the pivot index is found in case it yields an emergent pair.

This strategy turns out to carry over to the image setting as well.
The criterion used in Ripser for identifying the pivot index early is that its corresponding simplex appears in the filtration simultaneously with the simplex corresponding to the column.
When reducing the mixed basis boundary matrix $D^{\varphi}$ for $\varphi_{\bullet} \colon C_{\bullet}\to C'_{\bullet}$, we apply the criterion with respect to the filtration $C'$, which determines the row orders and hence the pivot of a column.
Note that the pairs identified this way do not necessarily have persistence $0$ anymore.

\subsection{Computational Experiments}\label{sec:benchmarks}
We provide an implementation \cite{ripser_image} of the algorithm resulting from \cref{thm:main} including the clearing optimization, based on the \texttt{simple} branch of Ripser \cite{MR4298669}, for the special case where $L_{\bullet} = \Rips_{\bullet}(X,d)$ and $K_{\bullet} = \Rips_{\bullet}(X,d')$ are filtrations of Vietoris--Rips complexes corresponding to two metrics $d$ and $d'$ on a finite set $X$ that satisfy $d(x,y) \geq d'(x,y)$ for all $x,y \in X$, with the map between filtrations given by the inclusions of $L_{t}$ into $K_{t}$. Recall that the inequality $d \geq d'$ ensures that $L_{t}$ is in fact a subcomplex of $K_{t}$.

We performed two example computations on a desktop computer with a 3.8 GHz 8-Core Intel Core i7 processor and 128 GB 2667 MHz DDR4 memory.
The first example is given by $X$ being 192 points sampled uniformly at random from the unit sphere in $\mathbb{R}^{3}$, with the distance $d$ being given by the geodesic distance on the sphere and the distance $d'$ being given by the Euclidean distance in $\mathbb{R}^{3}$.
The second example consists of 256 points sampled uniformly at random from $\mathrm{SO}(3)$, with $d$ given by the geodesic distance on $\mathrm{SO}(3) \cong \mathbb R P^3$ scaled by a factor of $\sqrt{2}$ and $d'$ given by the Frobenius norm distance on $\mathbb{R}^{3 \times 3}$. Scaling the geodesic distance is necessary to ensure that $d \geq d'$ holds.
Running times and memory usage are summarized in \cref{table:experiments}.

\begin{table}
\caption{Running time and memory usage of Ripser-image and Ripser-simple for different data sets.
The filtrations are defined by two difference metrics on the point cloud (extrinsic and intrinsic).
Ripser-image computes the image barcode, while Ripser-simple only computes the barcode of a single filtration.
The number of points in the data set is specified by $n$, the maximum dimension up to which persistence was computed is specified by $d$.
}
\label{table:experiments}
\begin{tabular}{l c c r@{ } r r@{ } r}
Data Set&$n$&$d$&\multicolumn{2}{c}{Ripser-image}&\multicolumn{2}{c}{Ripser-simple}\\
\hline
$S_2$&192&3&54.7\,s,&3.3\,GB& 31.4\,s,&3.3\,GB\\
$\mathrm{SO}(3)$&256&3&223\,s,&12.9\,GB&126\,s,&12.9\,GB
\end{tabular}
\end{table}

\section*{Acknowledgements}
This research has been supported by the German Research Foundation (DFG) through the Collaborative Research Center SFB/TRR 109 \emph{Discretization in Geometry and Dynamics}, the Collaborative Research Center SFB/TRR 191 \emph{Symplectic Structures in Geometry, Algebra and Dynamics}, the Cluster of Excellence EXC-2181/1 \emph{STRUCTURES}, and the Research Training Group RTG 2229 \emph{Asymptotic Invariants and Limits of Groups and Spaces}.

\printbibliography

\end{document}